\documentclass[12pt,a4paper,reqno]{amsart}
\usepackage[x11names]{xcolor}
\usepackage[linkcolor=Blue2,citecolor=red,colorlinks]{hyperref}
\usepackage[capitalize,nameinlink]{cleveref}
\usepackage[linkcolor=Blue2,citecolor=red,colorlinks]{hyperref}
\usepackage{color, bm, amscd, tikz-cd, amssymb}
\usepackage{color, bm, amscd, tikz-cd}

\usepackage[utf8]{inputenc}

\usepackage{array,amsbsy,amscd,amsfonts,color,amssymb,amstext,amsmath,latexsym,amsthm,amscd
,multicol,graphicx,tikz}
\usepackage[english]{babel}
\usepackage{hyperref}
\usepackage[mathscr]{eucal}
\usepackage{cleveref}

\newtheorem{theorem}{Theorem}
\numberwithin{theorem}{section}
\newtheorem{proposition}[theorem]{Proposition}
\newtheorem{lemma}[theorem]{Lemma}
\newtheorem{corollary}[theorem]{Corollary}

\theoremstyle{definition}
\theoremstyle{definition}\newtheorem{definition}[theorem]{Definition}
\theoremstyle{definition}\newtheorem{remark}[theorem]{Remark}
\theoremstyle{definition}\newtheorem{example}[theorem]{Example}
\theoremstyle{definition}
\theoremstyle{definition}
\theoremstyle{definition}
\theoremstyle{definition}

\newcommand{\id}{\operatorname{id}} 

\newcommand{\B}{\mathcal{B}}
\newcommand{\D}{\mathbb{D}}

\newcommand{\K}{\mathcal{K}}

\renewcommand{\H}{\mathcal{H}}

\newcommand{\F}{\mathcal{F}}

\renewcommand{\L}{\mathcal{L}}
\newcommand{\N}{\mathbb{N}}

\newcommand{\C}{\mathbb{C}}

\renewcommand{\S}{\mathcal{S}}

\newcommand{\Z}{\mathbb{Z}}

\newcommand{\E}{\mathcal{E}}

\newcommand{\bp}{\begin{proof}}
\newcommand{\ep}{\end{proof}}
\newcommand{\bdp}{\begin{dproof}}
\newcommand{\edp}{\end{dproof}}

\newcommand{\supp}{\operatorname{supp }}

\newcommand{\diag}{\operatorname{diag}}

\renewcommand{\ker}{\operatorname{ker }}

\newcommand{\sotlim}{\operatorname{SOT}-\!\!\lim}
\newcommand{\T}{\mathfrak T}

\begin{document}

\title{Isometric dilations and Beurling quotient modules}

\author[B.K. Das]{B. Krishna Das}
\address{Department of Mathematics, Indian Institute of Technology Bombay, Powai, Mumbai, 400076, India}
\email{bata436@gmail.com, dasb@math.iitb.ac.in}

\author[S. Rastogi]{Shubham Rastogi}
\address{Department of Mathematics, Indian Institute of Technology Bombay, Powai, Mumbai, 400076, India}
\email{shubhamrastogi4ever@gmail.com, shubham\_pd@isibang.ac.in}


	

\begin{abstract}
In \cite{Dan-Guo}, the authors provided a characterization of all sequences of doubly commuting pure contractions that have the sequence $\textbf{M}_z,$ the sequence of the multiplication by the coordinate functions on the vector-valued Hardy space $H^2_\E(\D_2^{(\infty)})$ over the Hilbert's multidisk $\D_2^{(\infty)},$ as their minimal isometric dilation for some Hilbert space $\E$. In this article, we remove the hypothesis of doubly commutativity and 
characterize all sequences of commuting contractions that are unitarily equivalent to $(P_{\mathcal Q}M_{z_i}|_{\mathcal Q})_{i\geq 1}$ for some quotient module $\mathcal Q$ of $H^2_\E(\D_2^{(\infty)})$.
We further characterize those sequences for which $\mathcal Q$ is a Beurling quotient module. This generalizes Theorem 3.2 from \cite{BKD} to a sequence of commuting contractions.
\end{abstract}

 \maketitle   

\section{Introduction}
The foundation of dilation theory is the result by Sz.-Nagy, which states that every contraction on a Hilbert space has a unique minimal isometric dilation. And\^o extended Sz.-Nagy's result to pairs of commuting contractions, showing that every such pair has an isometric dilation. A different proof of And\^o's dilation result surfaced recently with applications to sharp von Neumann inequality and distinguished varieties. For more details, see \cite{Das-Sarkar, DSS, BD, AM, Sau}. And\^o's dilation result does not extend to $n$-tuple of commuting contractions, in general, for $n\geq 3;$ see \cite{Parrott}. However, under specific additional positivity conditions, it is possible for such dilations to exist; see, for example, \cite{Brehmer, Nagy-Foias, Curto, Gasper-Suciu, BDHS, Popovici, Barik-Das, Popovici2}. One such condition is the positivity condition introduced by Brehmer in \cite{Brehmer}, known as Brehmer's positivity.  

In fact, it is now well-known through the work of Brehmer (\cite{Brehmer}), Halperin (\cite{Halperin}) and Sz.-Nagy (\cite{Nagy}) that an $n$-tuple $\textbf{T}$ satisfies Brehmer's positivity if and only if it admits a unique (up to unitary equivalence) minimal $*$-regular unitary dilation \cite[Chapter I, Theorem 9.1]{Nagy-Foias}. Restriction of a $*$-regular unitary dilation to a suitable subspace yields the existence of a minimal $*$-regular isometric dilation.  Moreover, it is also known that $\textbf T$ has a minimal $*$-regular isometric dilation if and only if it has a minimal doubly commuting isometric dilation (see \cite[Theorem 1]{Gasper-Suciu}; see also \cite[Theorem 2]{Timotin} and \cite{Skalski}). Since any $n$-tuple of doubly commuting pure isometries can be modeled by the tuple of multiplication by the coordinate functions  $\textbf{M}_z=(M_{z_1},M_{z_2},...,M_{z_n})$ on the vector-valued Hardy space over the polydisk $H^2_\E(\D^n),$ for some Hilbert space $\E,$ the equivalences above lead to one of the most significant results in multivariable dilation theory: an $n$-tuple of commuting pure contractions satisfies Brehmer's positivity if and only if it has $\textbf{M}_z$ on $H^2_\E(\D^n)$ as its minimal isometric dilation (see \cite{Curto, Muller}). Equivalently, this result provides a characterization of all $n$-tuples of commuting contractions that admit $\textbf{M}_z$ on $H^2_\E(\D^n)$ as their minimal isometric dilation, for some Hilbert space $\E.$

Brehmer's positivity condition regained interest in the eighties starting with the work of Alger (\cite{Agler}) in the setting of hypercontractions and subsequently generalized further in the setting of hypercontractions corresponding to multiweight sequences (see \cite{Curto, Curto2, Muller, O, BDDP}). Let $\textbf{T}=(T_n)_{n\in \mathbb{N}}$ be a sequence of commuting contractions on a Hilbert space $\H.$ The sequence $\textbf{T}$ is said to satisfy Brehmer's positivity if for each finite subset $F$ of $\mathbb N$, $\Delta_F(\mathbf T)\ge 0$,
where $$\Delta_F(\mathbf  T):=\sum_{G\subset F}(-1)^{|G|}T_G{T_G}^*$$ with $T_G:=T_{n_1}T_{n_2}\cdots T_{n_k}$ for any $G=\{n_1,n_2,...,n_k\}\subset F$ and by convention $T_{\emptyset}=I_\H.$ We note that, in the above definition of Brehmer's positivity, we adopt the convention of placing the adjoint on the right rather than on the left. This convention may cause some inconvenience for readers who are accustomed to the opposite convention. Immediate examples of Brehmer sequences of commuting contractions include sequences of doubly commuting contractions $\textbf{T}=(T_n)_{n\in\mathbb N}$, that is, $\mathbf T$ is a commuting sequence and for all $i\neq j$
\[
T_i^*T_j=T_jT_i^*.
\]
The starting point of this article is to study Brehmer's positivity of a sequence of commuting contractions in terms of their isometric dilation.

 Let $\Z_+^{(\infty)}$ denote the set of infinite tuples of nonnegative integers which are finitely supported. For a sequence of commuting contractions $\textbf{T}=(T_n)_{n\geq 1}$ on a Hilbert space $\H$ and $\alpha\in \Z_+^{(\infty)}$ with $\alpha=(\alpha_1,\alpha_2,...,\alpha_n,0,0,...)$, we denote by $\textbf{T}^\alpha$ the bounded operator $T_1^{\alpha_1}T_2^{\alpha_2}\cdots T_n^{\alpha_n}.$ With this notation, the notion of isometric (or unitary) dilation and $*$-regular isometric (or unitary) dilation for a sequence of commuting contractions are defined as follows.   
\begin{definition}\label{defn*-regular}
Let $\textbf{T}=(T_n)_{n\geq 1}$ be a sequence of commuting contractions on $\H.$ Let $\textbf{V}=(V_n)_{n\geq 1}$ be a sequence of commuting isometries (or unitaries) on a Hilbert space $\K$ which contains $\H$ as a closed subspace. Then $\textbf{V}$ is said to be an isometric (or unitary) dilation of $\textbf{T}$ if $$\textbf{T}^\alpha=P_{\H}\textbf{V}^\alpha|_\H$$ for all $\alpha \in \Z_+^{(\infty)},$ where $P_\H$ denotes the orthogonal projection of $\K$ onto $\H.$ 

If, in addition, $\textbf{V}$ satisfies $$\textbf{T}^\beta {\textbf{T}^*}^\alpha=P_\H {\textbf{V}^*}^\alpha\textbf{V}^\beta|_\H$$ for all $\alpha,\beta\in \Z_+^{(\infty)}$ with $\alpha \wedge \beta=(0,0,...),$ then $\textbf{V}$ is said to be a $*$-regular isometric (or unitary) dilation of $\textbf{T}$, where $\alpha \wedge \beta:=(min\{\alpha_1,\beta_1\},min\{\alpha_2,\beta_2\},...).$

An isometric dilation $\textbf{V}$ is said to be minimal if $$\overline{span}\{\textbf{V}^\alpha (h): \alpha \in \Z_+^{(\infty)}, h\in \H\}=\K.$$
\end{definition}

Minimal $*$-regular isometric dilation of a sequence of doubly commuting contractions is well-established (see \cite{Solel, Shalit}). In fact, it is also known that a sequence $\textbf T$ of commuting contractions has a $*$-regular unitary dilation if and only if $\textbf{T}$ satisfies Brehmer's positivity (\cite[Chapter I, Theorem 9.1]{Nagy-Foias}). In Section \ref{Section:1}, we extend this result by showing that a sequence $\textbf{T}$ satisfies Brehmer's positivity if and only if $\textbf T$ has a unique minimal $*$-regular isometric dilation if and only if $\textbf T$ has a unique minimal doubly commuting isometric dilation. Furthermore, each $T_i$ is a pure contraction, that is, \[T_i^{*n}\xrightarrow{n\to\infty} 0\ \text{in the strong operator topology for all}\ i\in\mathbb N,\] if and only if its minimal doubly commuting isometric dilation comprises  pure isometries. 


Let $\D_2^{(\infty)}$ be the Hilbert’s multidisk 
defined as $$\D_2^{(\infty)}:=\{z=(z_n)_{n\geq 1}\in l^2: |z_n|<1 \ for\ all\ n\in \N \}.$$ Let $\E$ be a Hilbert space. The analogue of the Hardy space in infinitely many variables is as follows.  
\begin{definition}\label{defnHardy}
The Hardy space $H^2_\E(\D_2^{(\infty)})$ over the Hilbert's multidisk  $\D_2^{(\infty)}$ is defined as 

$$H^2_\E(\D_2^{(\infty)}):=\Big\{\sum_{\alpha\in \Z_+^{(\infty)}}a_\alpha z^\alpha: a_\alpha\in \E, \sum_{\alpha\in \Z_+^{(\infty)}}\|a_\alpha\|^2<\infty\Big\}.$$    
\end{definition}
 Let $\textbf{M}_z=(M_{z_1},M_{z_2},...)$ denote the sequence of multiplication operators by the coordinate functions on the vector-valued Hardy space $H^2_\E(\D_2^{(\infty)})$ over the Hilbert's multidisk $\D_2^{(\infty)}$. While the sequence $\textbf{M}_z$ on $H^2_{\E}(\D_2^{\infty})$ consists of doubly commuting pure isometries, it is important to emphasize that, surprisingly, not every sequence $\textbf V$ of doubly commuting pure isometries can be modeled by $\textbf {M}_z$. This is one instance illustrating a genuine distinction between the behavior of infinitely many isometries and that of finitely many isometries.  The subtlety lies in the nature of the defect space of $\textbf V$: it may not be a generating wandering subspace for $\textbf V$ and can even be trivial in many cases. 
For further details, see the discussion preceding Lemma ~\ref{decomposition-DCPI} (see also \cite{Dan-Guo}). Therefore, to obtain a functional model for a sequence of commuting contractions, it is natural to ask the following question:

\textbf{Question:} {\em What are all the sequences $\textbf{T}$ of commuting contractions that have the sequence $\textbf{M}_z$ on $H^2_\E(\D_2^{(\infty)})$ as their minimal isometric dilation for some Hilbert space $\E$?}

A partial answer to the above question is obtained in \cite[Corollary 1.4]{Dan-Guo} under the additional hypothesis of doubly commutativity. Here, we obtain a complete answer to this question in Theorem \ref{CharacterizationBeurlingtype}. To broaden the applicability of this characterization, we also obtain an easily verifiable sufficient condition in Proposition~\ref{thm:summability}.
It is evident that a sequence of commuting contractions $\textbf T$ has the sequence $\textbf{M}_z$ on $H^2_\E(\D_2^{(\infty)})$ as its minimal isometric dilation for some Hilbert space $\E$ if and only if there exists a Hilbert space $\mathcal{F}$ such that $$\textbf T \cong P_{\mathcal Q}\textbf{M}_{z}|_{\mathcal Q}=(P_{\mathcal Q}M_{z_1}|_{\mathcal Q},P_{\mathcal Q}M_{z_2}|_{\mathcal Q}, \dots) $$ for some quotient module $\mathcal Q$ of $H^2_\F(\D_2^{(\infty)})$
(that is, a closed subspace $\mathcal Q$ of $H^2_\F(\D_2^{(\infty)})$ which is invariant under $M_{z_i}^*$ for all $i$). In such a case, it is expected that the quotient module $\mathcal Q$ should reflect the intrinsic properties of $\textbf T$ and vice versa. We investigate this relationship as described below.     

Let $\Omega$ denote either $\D_2^{(\infty)}$ or $\D^n$ for some $n\in \N.$ Let $H^{\infty}_{\mathcal B(\E_*,\E)}(\Omega)$ denote the set of all bounded holomorphic functions from  $\Omega$ to  $\B(\E_*,\E)$ (the set of all bounded linear operators from Hilbert spaces $\E_*$ to $\E$). A function $\Phi \in H^{\infty}_{\mathcal B(\E_*,\E)}(\Omega)$ is called inner if the multiplication operator $M_\Phi: H^2_{\E_*}(\Omega)\rightarrow H^2_\E(\Omega),$ given by $M_\Phi(f)=\Phi f,$ is an isometry. A quotient module $\mathcal Q$ of $H^2_\E(\Omega)$ is said to be a Beurling quotient module if there exist a Hilbert space $\E_*$ and an inner function $\Phi \in H^{\infty}_{\mathcal B(\E_*,\E)}(\Omega)$ such that $\mathcal Q^\perp=\Phi H^2_{\E_*}(\Omega).$ Note that when $\Omega=\D$, according to the Beurling-Lax-Halmos theorem, any quotient module of  $ H^2_{\E}(\D)$ is a Beurling quotient module. In stark contrast, when $\Omega=\D_2^{(\infty)}$ or $\D^n,$ where $n>1$, the scenario undergoes a significant change: typically, a quotient module in $H^2_\E(\Omega)$ does not necessarily possess a Beurling-type representation. For example, the quotient module $$\mathcal Q=\overline{span}\{1, z_1^n, z_2^n, z_1z_2: n\geq 1\}$$ in $H^2(\D_2^{(\infty)})$ does not possess a Beurling-type representation. This naturally leads to the following question:

\textbf{Question:} {\em What are all the sequences $\textbf T$ of commuting contractions that are unitary equivalent to $(P_{\mathcal Q}M_{z_i}|_{\mathcal Q})_{i\geq 1}$ for some Beurling quotient module $\mathcal Q$ of $H^2_\E(\D_2^{(\infty)})$?}

A concrete characterization of all such sequences in terms of some operator identity is obtained (see Theorem~\ref{CharacterizationdilationBeurlingquotientmodule} below),  which generalizes Theorem 3.2 from \cite{BKD} to a sequence of commuting contractions. This is obtained by first finding a characterization of Beurling quotient modules of $H^2_\E(\D_2^{(\infty)})$ in terms of compression operators (see Theorem \ref{compressionBeurlingquotientmodule} below). Along the way, we observe  that a submodule (that is, orthogonal complement of a quotient module) $\mathcal S$ of $H^2_\E(\D_2^{(\infty)})$ is doubly commuting, that is 
$$(M_{z_j}|_\S)^*(M_{z_i}|_\S)=(M_{z_i}|_\S)(M_{z_j}|_\S)^*\quad (i\neq j)$$
if and only  if $\mathcal S$ is a Brehmer submodule of $H^2_\E(\D_2^{(\infty)})$, that is the sequence of isometries $\textbf{M}_z|_{\mathcal S}=(M_{z_1}|_{\mathcal S},M_{z_2}|_{\mathcal S},...)$ satisfies Brehmer's positivity. This extends the corresponding result for submodules of $H^2_{\mathcal E}(\mathbb D^n)$ obtained in \cite[Theorem 7.6]{Rani}. Furthermore, the approach used in our proof is different and provides a simpler proof even for the case of $H^2_{\mathcal E}(\mathbb D^n)$.

The article is organized as follows. In the next section, we show the equivalence of Brehmer's positivity and the existence of minimal doubly commuting isometric dilation. A characterization of sequences of commuting contractions which are unitarily equivalent to $P_{\mathcal Q}\textbf{M}_{z}|_{\mathcal Q}$ for some quotient module $\mathcal Q$ of $H^2_\E(\D_2^{(\infty)})$ is obtained in Section~\ref{Section:2}. In the last section, we characterize all such sequences for which the model space $\mathcal Q$ is a Beurling quotient module.

\section{Brehmer's positivity and minimal doubly commuting isometric dilation}\label{Section:1}
A sequence of doubly commuting contractions certainly satisfies Brehmer's positivity. However, a sequence of commuting contractions satisfying Brehmer's positivity is not doubly commuting, in general. For instance, consider $$\mathcal Q=\overline{span}\{1, z_1z_2, z_i^n: i,n\geq 1\}\subset H^2(\D_2^{(\infty)}),$$ 
and for $i\in \N,$ define $$T_i:=P_\mathcal QM_{z_i}|_\mathcal Q.$$ Then, as $\mathcal Q$ is a quotient module, for any finite subset $F$ of $\mathbb N$, 
\[
\Delta_F(\mathbf T)= 
P_{\mathcal Q}\Delta_F(\textbf{M}_{z})P_{\mathcal Q}\ge 0.
\]
Therefore, $\textbf T=(T_i)_{i\geq 1}$ satisfies Brehmer's positivity. But it is not doubly commuting as
\[
P_{\mathcal Q}M_{z_1}M_{z_2}^*(z_1z_2)=z_1^2\neq 0 =  M_{z_2}^*P_{\mathcal Q}M_{z_1}(z_1z_2).
\]
Hence, the set of all sequences of commuting contractions that satisfy Brehmer's positivity is strictly larger than the set of all sequences of doubly commuting contractions. It is a well-established result that any sequence of doubly commuting contractions has a minimal $*$-regular isometric dilation, which is unique up to unitary equivalence. In addition, the dilation consists of doubly commuting isometries (see \cite[Theorem 3.10]{Solel}, \cite[Theorem 4.2]{Shalit}). In this section, we extend this result to a sequence of commuting contractions satisfying Brehmer's positivity. A sequence $\textbf{T}=(T_n)_{n\geq 1}$ of commuting contractions on a Hilbert space $\H$ is said to be pure if each $T_n$ is pure. The equivalence $(1)\iff (2)$ in the following theorem follows immediately from \cite[Chapter I, Theorem 9.1]{Nagy-Foias}. To establish $(2)\implies(3)$ and the uniqueness assertion, we use ideas from the proofs of \cite[Theorem 3.10]{Solel} and \cite[Proposition 3.7]{Solel}, respectively. We also repeatedly use the standard fact that a minimal isometric dilation of a commuting tuple or sequence of contractions is a co-extension.

\begin{theorem}\label{*-regular dilation}
Let $\textbf{T}=(T_n)_{n\geq 1}$ be a sequence of commuting contractions on a Hilbert space $\H$. Then the following are equivalent:

\begin{enumerate}
\item[(1)] 
$\textbf{T}$ satisfies Brehmer's positivity.

\item[(2)]
$\textbf{T}$ has a minimal $*$-regular isometric dilation.

\item[(3)]
$\textbf{T}$ has a minimal doubly commuting isometric dilation.
\end{enumerate}

Moreover, if $\textbf{T}$ satisfies any of the above conditions, then the minimal doubly commuting isometric dilation is unique up to unitary equivalence. Furthermore, it is pure if and only if  $\textbf{T}$ is pure.  
\end{theorem}

\begin{proof}
$(1)\implies(2).$
Suppose $\textbf{T}$ satisfies Brehmer's positivity. Then a $*$-regular unitary dilation $\textbf{U}=(U_n)_{n\geq 1}$ for the sequence $\textbf{T}$ exists; see \cite[Chapter I, Theorem 9.1]{Nagy-Foias}. Let $\tilde{\K}:= \overline{span}\{\textbf{U}^\alpha(h):h\in \H, \alpha \in \Z_+^{(\infty)}\}.$ Then the sequence $\textbf{V}=(V_n)_{n\geq 1}$ is a minimal $*$-regular isometric dilation of $\textbf{T}$ where $V_n:=U_n|_{\tilde \K}.$

$(2)\implies(1).$ Suppose $\textbf{V}=(V_n)_{n\geq 1}$ is a minimal $*$-regular isometric dilation  of $\textbf{T}$. Let $\textbf{U}=(U_n)_{n\geq 1}$ be the minimal unitary extension of $\textbf{V}.$ Then $\textbf{U}$ is a minimal $*$-regular unitary dilation  of $\textbf{T}$. Hence, by \cite[Chapter I, Theorem 9.1]{Nagy-Foias}, $\textbf{T}$ satisfies Brehmer's positivity.

$(2)\implies(3).$ Suppose $\textbf{T}$ has a minimal $*$-regular isometric dilation $\textbf{V}.$ We show that $\textbf{V}$ consists of doubly commuting isometries. Let $i,j \in \N$ with $i\neq j.$ Suppose $\alpha, \beta \in \Z_+^{(\infty)}$ and $h_1,h_2\in \H.$ Note that \begin{align}\label{simplification}\langle V_j^*V_i (\textbf{V}^\alpha h_1), (\textbf{V}^\beta h_2)\rangle=\langle {\textbf{V}^*}^{\beta+e_j}\textbf{V}^{\alpha +e_i}(h_1), h_2  \rangle.\end{align}
If $\alpha_j \geq 1$, 
where $\alpha=(\alpha_k)_{k\geq 1}$,
then 
\begin{align}\label{doublycommuting}
\langle V_iV_j^* (\textbf{V}^\alpha h_1), \textbf{V}^\beta h_2\rangle=\langle {\textbf{V}^*}^{\beta+e_j}\textbf{V}^{\alpha +e_i}h_1, h_2  \rangle=\langle V_j^*V_i (\textbf{V}^\alpha h_1), \textbf{V}^\beta h_2\rangle.
\end{align} 
The case where $\beta_i\ge 1$ can be dealt with in a similar way. 
Now suppose that $\alpha_j=0=\beta_i.$ To demonstrate equality ~\eqref{doublycommuting} in this case, we require the following claim.

\textbf{Claim:} Let $\alpha \in \Z_+^{(\infty)}$ with $\alpha_j=0.$ Then $V_j^*\textbf{V}^\alpha (h)=\textbf{V}^\alpha V_j^*(h)$ for all $h \in \H.$

\textit{Proof of the claim:} Let $\gamma \in \Z_+^{(\infty)}$ and $\tilde{h} \in \H$.
Then \begin{align*}
\langle V_j^*\textbf{V}^\alpha (h), \textbf{V}^\gamma (\tilde{h})\rangle
&=\langle P_\H {\textbf{V}^*}^{\gamma +e_j}\textbf{V}^\alpha (h), \tilde {h}\rangle\\
&=\langle P_\H {\textbf{V}^*}^{\gamma^{'}}\textbf{V}^{\alpha^{'}} (h), \tilde {h}\rangle\\
&=\langle \textbf{T}^{\alpha^{'}}{\textbf{T}^*}^{\gamma^{'}}(h), \tilde{h}\rangle 
\end{align*} where $\gamma^{'}=\gamma+e_j- ((\gamma+e_j) \wedge \alpha)$ and $\alpha {'}=\alpha-((\gamma+e_j)\wedge \alpha)$ and 
\begin{align*}
\langle \textbf{V}^\alpha V_j^{*}(h), \textbf{V}^{\gamma}(\tilde h)\rangle
&=\langle {\textbf{V}^*}^{\gamma}\textbf{V}^{\alpha}T_j^*(h), \tilde{h}\rangle\\
&=\langle P_\H {\textbf{V}^*}^{\gamma^{''}}\textbf{V}^{\alpha^{''}} (T_j^*h), \tilde {h}\rangle\\
&=\langle \textbf{T}^{\alpha^{''}}{\textbf{T}^*}^{\gamma^{''}+e_j}(h), \tilde{h}\rangle
\end{align*} 
where $\gamma^{''}=\gamma- (\gamma \wedge \alpha)$ and $\alpha^ {''}=\alpha-(\gamma\wedge \alpha)$

The proof of the claim follows by noting that $\gamma'' + e_j = \gamma'$ and $\alpha'' = \alpha',$ as $\alpha_j = 0.$

Now we are ready to prove equality \eqref{doublycommuting} for the case where $\alpha_j=0=\beta_i$, using the aforementioned claim. To this end, note that \begin{align*}
\langle V_j^*(\textbf{V}^\alpha h_1), V_i^*(\textbf{V}^\beta h_2)\rangle
&=\langle \textbf{V}^\alpha V_j^*h_1, \textbf{V}^\beta V_i^*h_2\rangle\\
&=\langle \textbf{V}^\alpha T_j^*h_1, \textbf{V}^\beta T_i^*h_2\rangle\\
&=\langle P_{\H}{\textbf{V}^*}^{\beta ^{'}}{\textbf{V}}^{\alpha ^{'}}T_j^*h_1, T_i^*h_2\rangle\\
&=\langle \textbf{T}^{\alpha^{'}+e_i}{\textbf{T}^*}^{\beta ^{'}+e_j} h_1,  h_2\rangle 
\end{align*} where $\beta{'}=\beta-(\alpha\wedge \beta)$ and $\alpha{'}=\alpha-(\alpha\wedge \beta)$ and 

\begin{align*}
\langle {\textbf{V}^*}^{\beta+e_j}\textbf{V}^{\alpha+e_i}(h_1),h_2\rangle
&=\langle P_{\H}{\textbf{V}^*}^{\beta ^{''}}\textbf{V}^{\alpha^{''}}(h_1),h_2\rangle\\
&=\langle \textbf{T}^{\alpha^{''}}{\textbf{T}^*}^{\beta ^{''}}(h_1),h_2\rangle
\end{align*} 
where $\beta^{''}=\beta+e_j-((\alpha+e_i)\wedge (\beta+e_j))$ and $\alpha^{''}=\alpha+e_i-((\alpha+e_i)\wedge (\beta+e_j))$

The doubly commutativity of $\textbf{V}$ now follows by observing that $\alpha' + e_i = \alpha''$ and $\beta' + e_j = \beta''$ (due to the case where $\alpha_j = 0 = \beta_i$) and by referencing \cref{simplification}, along with the fact that the dilation is minimal.

$(3)\implies(2).$ Let $\textbf{V}$ be a minimal doubly commuting isometric dilation of $\textbf{T}.$ Then, using the doubly commutativity of $\textbf{V}$ and its minimality, we get that $\textbf{V}$ is a $*$-regular isometric dilation. 

Thus, we have established the equivalences. Next, suppose $\textbf{T}$ satisfies any of the equivalent conditions in the statement of the theorem. We show that its minimal doubly commuting isometric dilation is unique up to  unitary equivalence. Let $\textbf{V}=(V_n)_{n\geq 1}$ on $\K$ and $\textbf{W}=(W_n)_{n\geq 1}$ on $\K^{'}$ be two minimal doubly commuting isometric dilations of $\textbf{T}.$ Then, $\textbf{V}$ and $\textbf{W}$ are two minimal $*$-regular isometric dilations of $\textbf{T}$ (see $(3)\implies(2)$ above). For any $\alpha, \beta \in \Z_+^{(\infty)}$ and $h,\tilde{h}\in \H,$ we have 
\begin{align*}
\langle \textbf{V}^\alpha(h), \textbf{V}^\beta(\tilde{h})\rangle
&=\langle P_\H{\textbf{V}^*}^\beta\textbf{V}^\alpha (h), \tilde{h}\rangle\\
&=\langle P_\H{\textbf{V}^*}^{\beta{'}} \textbf{V}^{\alpha{'}} (h), \tilde{h}\rangle\\
&=\langle \textbf{T}^{\alpha{'}}{\textbf{T}^*}^{\beta{'}}(h),\tilde{h}\rangle\\
&=\langle \textbf{W}^\alpha(h), \textbf{W}^\beta(\tilde{h})\rangle,
\end{align*}
where $\beta{'}=\beta- (\alpha \wedge \beta)$  and $\alpha{'}=\alpha -(\alpha \wedge \beta)$ (note that $\alpha{'}\wedge \beta{'}=(0,0,...)$). Therefore, we obtain the well-defined surjective isometry $$\Gamma: span\{\textbf{V}^\alpha(h): \alpha \in \Z_+^{(\infty)}, h\in \H\}\rightarrow span\{\textbf{W}^\alpha(h): \alpha \in \Z_+^{(\infty)}, h\in \H\}$$ given by $\Gamma(\textbf{V}^\alpha h)=\textbf{W}^\alpha h$ for $\alpha \in \Z_+^{(\infty)}$ and $h\in \H.$ Hence, $\Gamma$ extends to a unitary from $\K$ to $\K^{'}$.  It is easy to verify that $\Gamma V_n\Gamma^*=W_n$ for $n\geq 1,$ and $\Gamma(h)=h$ for all $h\in \H.$ This shows that the minimal doubly commuting isometric dilation is unique up to  unitary equivalence. 

For the furthermore part, let $\textbf{V}$ be the minimal  doubly commuting isometric dilation of $\textbf{T}.$ Since $\textbf{V}$ is minimal, the space $\H$ is invariant under $V_n^*,$ and $V_n^*|_{\H}=T_n^*$ for all $n\in \N.$ Hence, $\textbf{T}$ is pure if $\textbf{V}$ is pure.

Conversely, suppose $\textbf{T}$ is pure. Then using the fact that $\textbf{V}$ is doubly commuting and $V_n^*|_{\H}=T_n^*$ for all $n\in \N,$ we get that for $m>\alpha_n$ and $h\in \H$, $$\|{V_n^*}^m(\textbf{V}^\alpha h)\|=\|{T_n^*}^{m-\alpha_n}h\|.$$ Hence, by the minimality, we get that $\textbf{V}$ is pure.
This completes the proof. 
\end{proof}
The following remark follows from the proof of the above theorem.
\begin{remark}\label{*-dc}
Let $\textbf{T}=(T_n)_{n\geq 1}$ be a sequence of commuting contractions on a Hilbert space $\H$.
\begin{enumerate}
\item A sequence $\textbf{V}=(V_n)_{n\geq 1}$ is
a minimal $*$-regular isometric dilation of $\textbf{T}$ if and only if it is a minimal doubly commuting isometric dilation of $\textbf{T}$.

\item The minimal $*$-regular isometric dilation of $\textbf{T}$ is unique up to  unitary equivalence, if it exists.
\end{enumerate} 
\end{remark}


Since any sequence of doubly commuting contractions satisfies Brehmer's positivity, the following corollary, established in \cite[Theorem 3.10]{Solel}, is a straightforward consequence of Theorem \ref{*-regular dilation} and Remark \ref{*-dc}. 
\begin{corollary} Let $\textbf{T}$ be a sequence of doubly commuting contractions on a Hilbert space. Then it has a minimal $*$-regular isometric dilation $\textbf{V},$  which is unique up to unitary equivalence. Moreover, $\textbf{V}$ consists of doubly commuting isometries.

Furthermore,
$\textbf{T}$ is pure if and only if $\textbf{V}$ is pure. 
\end{corollary}

\section{Analytic model for sequence of commuting contractions}\label{Section:2}
In \cite[Corollary 1.4]{Dan-Guo}, the authors provided a characterization of all sequences of doubly commuting pure contractions that have $\textbf{M}_z$ on $H^2_\E(\D_2^{(\infty)})$  as their minimal isometric dilation for some Hilbert space $\E$. In this section, we characterize all sequences of commuting contractions that have the sequence $\textbf{M}_z$ on $H^2_\E(\D_2^{(\infty)})$ as their minimal isometric dilation for some Hilbert space $\E$. This improves the result in \cite{Dan-Guo}, as it removes the assumption that the sequence consists of doubly commuting contractions. To address this, we introduce the concept of the defect operator for sequences of commuting contractions that satisfy Brehmer's positivity. This definition extends the previously established notion of the defect operator for sequences of doubly commuting contractions.

Let $\textbf{T}=(T_n)_{n\geq 1}$ be a sequence of commuting contractions on $\H$ that satisfies Brehmer's positivity. Set \[
[n]=\{1,\ldots,n\}, \text{ with } [0]=\emptyset.\] 
Then with the above notation, $(\Delta_{[n]}(\mathbf T))_{n\ge 1}$ is a decreasing sequence of positive operators because $$\Delta_{[n]}(\mathbf T)-\Delta_{[n+1]}(\mathbf T)=T_{n+1}\Delta_{[n]}(\mathbf T)T_{n+1}^*  
$$ for all $n\in\mathbb N$. Consequently, the sequence $(\Delta_{[n]}(\mathbf T))_{n\geq 1}$ is a decreasing sequence of positive operators and therefore has a strong operator limit. Define $$D_{\textbf{T}^*}^2:=\sotlim_{n\to \infty}\Delta_{[n]}(\mathbf T).$$

We call the positive square root of $D_{\textbf{T}^*}^2$ the \textit{defect operator} of $\textbf{T}$ and denote it by $D_{\textbf{T}^*}.$ The \textit{defect space} $\mathcal D_{\textbf{T}^*}$ of $\textbf{T}$ is defined as $\mathcal D_{\textbf{T}^*}:=\overline{\text{Ran}}D_{\textbf{T}^*}$. It should be noted that the above limit is independent of the enumeration of the sequence $\mathbf T$ since the finite subsets form a directed set and every increasing finite exhaustion is cofinal. If the sequence $\textbf T$ is doubly commuting then 
$$D_{\textbf{T}^*}^2=\sotlim_{n\to \infty}(I-T_1T_1^*)\cdots(I-T_nT_n^*).$$
Thus, the notion of defect operator, as defined earlier, coincides with that in \cite[page 305]{Dan-Guo} when restricted to the class of doubly commuting contractions. For any finite subset $F\subset\mathbb N$ and $j\notin F$, we also have
\begin{equation}\label{monotonicity}
 \Delta_{F\cup\{j\}}(\mathbf T)=\Delta_F(\mathbf T)-T_j\Delta_F(\mathbf T)T_j^*.
\end{equation}
Consequently, the finite defects decrease under inclusion and lie between
$0$ and $I$. This alow us to define 
\[
\T_{(n}(\mathbf T):=\sotlim_{m\to\infty}
          \Delta_{\{n+1,\ldots,m\}}(\mathbf T)\quad(n\ge0).\]
We call $\T_{(n}(\mathbf T)$ as the \textit{tail defect operator}. Monotonicity of the finite defects
also gives $0\le \T_{(n}(\mathbf T)\le \T_{(n+1}(\mathbf T)\le I$.
The following lemma determines the tail defect operators and their strong operator limit, which we will need in the sequel. The support of an element $\alpha\in \Z_+^{(\infty)}$ is defined as $\supp\alpha:=\{ i\in \mathbb N: \alpha_i>0\}$. 
\begin{lemma}\label{tail defect}
Let $\mathbf T$ be a sequence of commuting pure contractions on $\mathcal H$ satisfying Brehmer's positivity. Then for every $n\ge0$,
\begin{equation}\label{eq:partial}
 \T_{(n}(\mathbf T)=\sum_{\substack{\alpha\in\Z_+^{(\infty)}\\ \supp\alpha\subseteq[n]}}
        T^\alpha D_{\mathbf T^*}^2T^{*\alpha}, \text { and }  \sotlim_{n\to\infty}\T_{(n}(\mathbf T)=\sum_{\alpha\in\Z_+^{(\infty)}}T^\alpha D_{\mathbf T^*}^2T^{*\alpha},
\end{equation}
where the sum converges in the strong operator topology. 
\end{lemma}

\begin{proof}
For each $i\in\mathbb N$, define maps on $\mathcal B(\mathcal H)$ by $\Phi_i(X)=T_iXT_i^*$. Then it is easy to see that $\Phi_i$ and $\Phi_j$ commutes with each other. We use $\id$ for the identity map on $\mathcal B(\mathcal H)$.
The repeated application of \cref{monotonicity} shows that, for $m\ge n+1$,
\[
 \Delta_{[m]}(\mathbf T)
 =\Big(\prod_{i=1}^{n}(\id-\Phi_i)\Big)
       \Delta_{\{n+1,\ldots,m\}}(\mathbf T).
\]
Since each $\Phi_i$ preserves strong limits, passing to the strong
limit in the above finite product gives
\begin{equation*}
 D_{\mathbf T^*}^2=\Big(\prod_{i=1}^{n}(\id-\Phi_i)\Big)\T_{(n}(\mathbf T).
\end{equation*}

Fix $n\ge1$. For an integer $r\ge1$, the finite geometric-sum identity
and commutativity of the maps yield
\begin{align}
 {\sum_{\substack{\supp\alpha\subseteq[n]\\
                         0\le\alpha_i<r\ (1\le i\le n)}}}
        T^\alpha D_{\mathbf T^*}^2T^{*\alpha}
 &=\prod_{i=1}^{n}\left(\sum_{k=0}^{r-1}\Phi_i^k\right)
      \Big (\prod_{i=1}^{n}(\id-\Phi_i)\Big )\T_{(n}(\mathbf T)\notag\\
 &=\Big (\prod_{i=1}^{n}(\id-\Phi_i^r)\Big)\T_{(n}(\mathbf T)\notag\\
 &=\T_{(n}(\mathbf T)+\sum_{\emptyset\ne G\subseteq[n]}
                  (-1)^{|G|}T_G^r\T_{(n}(\mathbf T)T_G^{*r}.
 \label{eq:telescoping}
\end{align}
Every nonempty-set term on the last line tends strongly to zero as $r\longrightarrow \infty$.
Indeed, fix $i\in G$.
Note that for every non-empty subset $G$, $T_G^r$ is a contraction and if $i\in G$ then for $h\in\mathcal H$ we have
\[
 \|T_G^r\T_{(n}(\mathbf T)T_G^{*r}h\|
 \le\|\T_{(n}(\mathbf T)\|\,\|T_G^{*r}h\|
 \le\|\T_{(n}(\mathbf T)\|\,\|T_i^{*r}h\|
 \longrightarrow0.
\]
The convergence follows as $T_i$ is a pure contraction. There are only finitely many subsets $G\subseteq[n]$,
so taking limit as $r\to \infty$, \cref{eq:telescoping} proves the first identity in \cref{eq:partial}, and the second identity is an easy consequence of the first identity. This completes the proof.


\end{proof}


Let $\textbf{V}=(V_n)_{n\geq 1}$ be a sequence of doubly commuting pure isometries on $\H$. Then, it is not hard to verify that the defect space $\mathcal D_{\textbf{V}^*}=\cap_{n\in \N} \ker{V_n}^*,$ and $\mathcal D_{\textbf{V}^*}$ is a wandering subspace for $\textbf{V}.$ The sequence $\textbf{V}$ is said to be of Beurling's type if $\mathcal D_{\textbf{V}^*}$ is a generating wandering subspace, that is $$\bigoplus_{\alpha \in \Z_+^{(\infty)}} \textbf{V}^{\alpha} (\mathcal D_{\textbf{V}^*})=\H.$$

It is easy to see that $\textbf{M}_z$ on $H^2_\E(\D_2^{(\infty)})$  is of Beurling's type. In fact, $\textbf{M}_z$ serves as a model for sequences of doubly commuting pure isometries which are of Beurling's type. If $\textbf{V}$ is a finite tuple of doubly commuting pure isometries, it always exhibits Beurling's type property. However, this does not hold for a sequence of doubly commuting pure isometries, in general. Indeed, in \cite{Dan-Guo}, the authors provided numerous examples for which $\mathcal D_{\textbf{V}^*}=\{0\}$. We observe the following decomposition for a sequence of doubly commuting pure isometries.

\begin{lemma}\label{decomposition-DCPI}
Let $\textbf{V}$ be a sequence of doubly commuting pure isometries acting on $\H.$ Then, the space $\H$ decomposes into the direct sum $\H=\H_\B\oplus H_0$ of two joint-reducing subspaces of $\textbf{V}$ such that  $\textbf{V}|_{\H_\B}$ is of Beurling's type and  $\mathcal D_{(\textbf{V}|_{\H_0})^*}=\{0\}.$ 
\end{lemma}
\begin{proof}
Set 
\[\H_\B:=\bigoplus_{\alpha \in \Z_+^{(\infty)}} \textbf{V}^{\alpha} (\mathcal D_{\textbf{V}^*}), \text{ and } 
 \H_0= {\H_\B}^\perp.
 \] 
 It is then routine to check that $\textbf{V}|_{\H_\B}$ is of Beurling's type and  $\mathcal D_{(\textbf{V}|_{\H_0})^*}=\{0\}$.
\end{proof}

According to Theorem \ref{*-regular dilation}, a sequence $\textbf{T}$ that satisfies Brehmer's positivity has the unique minimal doubly commuting isometric dilation $\textbf{V}$. Next, we delineate all sequences $\textbf{T}$ for which $\textbf{V}$ is of Beurling's type. Certain steps in the proof of the following proposition are motivated by the proof of \cite[Lemma 2.6]{Dan-Guo}.

\begin{proposition}\label{Beurlingtype*regulardilation}
Let $\textbf{T}=(T_n)_{n\geq 1}$ be a sequence of commuting contractions on $\H$ that satisfies Brehmer's positivity. Let $\textbf{V}$ be the minimal doubly commuting isometric dilation of $\textbf{T}$. 
Then $\textbf{V}$ is of Beurling's type if and only if $$\|x\|^2=\sum_{\alpha \in \Z_+^{(\infty)}}\|D_{\textbf{T}^*}{\textbf{T}^*}^\alpha x\|^2$$ for all $x\in \H.$
\end{proposition}
\begin{proof}
To show the forward direction, we first show that $\|D_{\textbf{T}^*}x\|=\|D_{\textbf{V}^*}x\|$ for all $x\in \H.$ To this end, note that $\|D_{\textbf{T}^*}x\|=\lim_{n\to \infty}\| \Delta_{[n]}(\mathbf T)^{\frac{1}{2}}x\|$ and $\|D_{\textbf{V}^*}x\|=\lim_{n\to \infty}\| \Delta_{[n]}(\mathbf V)^{\frac{1}{2}}x\|.$ Therefore, it suffices to prove that for $n\in \N$ and $x\in \H,$  $\| \Delta_{[n]}(\mathbf T)^{\frac{1}{2}}x\|=\| \Delta_{[n]}(\mathbf V)^{\frac{1}{2}}x\|.$ The identity is obviously true for $n=1$ as $\textbf V$ is a coextension of $\textbf T$ by minimality. Suppose the equality is true for $n=k.$ Then, \begin{align*}
\|\Delta_{[k+1]}(\mathbf T)^{\frac{1}{2}}x\|^2
&=\|\Delta_{[k]}(\mathbf T)^{\frac{1}{2}}x\|^2-\|\Delta_{[k]}(\mathbf T)^{\frac{1}{2}}T_{k+1}^*x\|^2\\
&=\|\Delta_{[k]}(\mathbf V)^{\frac{1}{2}}x\|^2-\|\Delta_{[k]}(\mathbf V)^{\frac{1}{2}}T_{k+1}^*x\|^2\\
&=\|\Delta_{[k]}(\mathbf V)^{\frac{1}{2}}x\|^2-\|\Delta_{[k]}(\mathbf V)^{\frac{1}{2}}V_{k+1}^*x\|^2\ (\text{by minimality})\\
&=\|\Delta_{[k+1]}(\mathbf V)^{\frac{1}{2}}x\|^2.
\end{align*}Therefore, we have shown that $\|D_{\textbf{T}^*}x\|=\|D_{\textbf{V}^*}x\|$ for all $x\in \H.$ Hence, again using the fact that $\textbf{V}$ is coextension of $\textbf T$,  we get $$\sum_{\alpha \in \Z_+^{(\infty)}}\|D_{\textbf{T}^*}{\textbf{T}^*}^\alpha x\|^2=\sum_{\alpha \in \Z_+^{(\infty)}}\|D_{\textbf{T}^*}{\textbf{V}^*}^\alpha x\|^2=\sum_{\alpha \in \Z_+^{(\infty)}}\|D_{\textbf{V}^*}{\textbf{V}^*}^\alpha x\|^2$$
for all $x\in \H.$

Now, suppose $\textbf{V}$ is of Beurling's type. Then, it is easy to see that $\sum_{\alpha \in \Z_+^{(\infty)}}\|D_{\textbf{V}^*}{\textbf{V}^*}^\alpha x\|^2=\|x\|^2$ (see also \cite[Corollary 1.4]{Dan-Guo}). Hence, the proof of the forward direction is complete.

Conversely, suppose $\|x\|^2=\sum_{\alpha \in \Z_+^{(\infty)}}\|D_{\textbf{T}^*}{\textbf{T}^*}^\alpha x\|^2$ for all $x\in \H.$ Define $$\Pi:\H\rightarrow H^2_{\mathcal D_{\textbf{T}^*}}(\D_2^{(\infty)})$$by, $$\Pi(x):=\sum_{\alpha \in \Z_+^{(\infty)}}(D_{\textbf{T}^*}{\textbf{T}^*}^\alpha x)z^\alpha.$$

Clearly, $\Pi$ is an isometry. Note that \begin{align}\label{invariance}
\Pi T_i^*=M_{z_i}^*\Pi.
\end{align}  
Hence, the sequence $\textbf{M}_z$ on $H^2_{\mathcal D_{\textbf{T}^*}}(\D_2^{(\infty)})$ is a coextension of $\textbf{T}.$ Now consider $$\L:=\overline{span}\{\textbf{M}_z^\alpha(\Pi(\H)):\alpha \in \Z_+^{(\infty)}\}.$$ Since $\Pi(\H)$ is invariant under ${\textbf{M}_z}^*$ (by  \cref{invariance}) and $\textbf{M}_z$ is a doubly commuting sequence of isometries, we get that the space $\L$ is reducing for $\textbf{M}_z.$ Therefore, by \cite[Lemmas 2.11 and 2.10]{Dan-Guo}, $\L=H^2_{\mathcal E} (\D_2^{(\infty)})$ for some closed subspace $\E$ of $\mathcal D_{\textbf{T}^*}.$ Let $P_0\in \mathcal B (H^2_{\mathcal D_{\textbf{T}^*}}(\D_2^{(\infty)}))$ be the projection onto constant functions. Then 
\[
P_0=\sotlim_{n\to \infty}\prod_{j=1}^n (I-M_{z_j}M_{z_j}^*), \text{ and } P_0(\Pi x)= D_{\textbf T^*}x.\ \quad (x\in\mathcal H)
\]
Since $\L$ is a reducing subspace, $P_0(\L)\subseteq \L $. In particular, $\mathcal D_{\textbf T^*}\subseteq \L$. This shows that $\L=H^2_{\mathcal D_{\textbf{T}^*}}(\D_2^{(\infty)}) $.  Hence, $\textbf{M}_z$ on $H^2_{\mathcal D_{\textbf{T}^*}}(\D_2^{(\infty)})$ is the minimal doubly commuting isometric dilation of $\textbf{T}.$ Therefore, the minimal doubly commuting isometric dilation of $\textbf{T}$ is of Beurling's type.
\end{proof}

We need the following useful lemma for the main result of this section. 
\begin{lemma}\label{auto_pure}
Let $\textbf{T}$ be a sequence of commuting contractions on a Hilbert space $\H$ satisfying Brehmer's positivity. 
If for all $x\in\mathcal H$, 
\begin{equation}\label{identity}\|x\|^2=\sum_{\alpha \in \Z_+^{(\infty)}}\|D_{\textbf{T}^*}{\textbf{T}^*}^\alpha x\|^2,
\end{equation}
then $\mathbf T$ is pure, that is, $T_i$ is a pure contraction for all $i$.
\end{lemma}
\begin{proof}
Let $i\in\mathbb N$. To show, $T_i$ is a pure contraction, we replace $x$ by $T_i^{* m}x$ in eq.~\eqref{identity}. We obtain \begin{align*}
\|T_i^{* m}x\|^2 &=\sum_{\alpha \in \Z_+^{(\infty)}}\|D_{\textbf{T}^*}{\textbf{T}^*}^\alpha T_i^{* m} x\|^2\\
& = \sum_{\substack{\beta\in\mathbb{Z}_+^n\\ \beta_i\geq m}}
\left\|D_{\mathbf T^*}\mathbf T^{*\beta}x
\right\|^2.
\end{align*}
The last expression is a tail of the summable series \eqref{identity} as each finitely supported multi-index in $\mathbb Z_+^{(\infty)}$ has finite $(i)$-th coordinate. Hence, it follows that $T_i^{* m}x\to 0$ as $m\to \infty$. This completes the proof. 
\end{proof}
The equivalence between $(1)$ and $(2)$ in the following theorem follows from Lemma~\ref{tail defect} and Lemma~\ref{auto_pure}. 
Whereas, the equivalence between $(2)$ and $(3)$ is a consequence of Theorem \ref{*-regular dilation} and Proposition \ref{Beurlingtype*regulardilation}. The equivalence between $(3)$ and $(4)$ is immediate.
\begin{theorem}\label{CharacterizationBeurlingtype}
Let $\textbf{T}$ be a sequence of commuting contractions on a Hilbert space $\H.$ The following statements are equivalent:

\begin{enumerate}
\item[(1)] $\mathbf T$ is pure, satisfies Brehmer's positivity and 
\[
\sotlim_{n\to \infty}\T_{(n}(\mathbf T)= I_{\mathcal H}.
\]
\item[(2)] $\textbf{T}$ satisfies Brehmer's positivity, and for all $x\in \H$,
$$\|x\|^2=\sum_{\alpha \in \Z_+^{(\infty)}}\|D_{\textbf{T}^*}{\textbf{T}^*}^\alpha x\|^2. $$ 

\item[(3)] $\textbf{T}$ has a minimal doubly commuting isometric dilation which is of Beurling's type.

\item[(4)] $\textbf{T}$ has $\textbf{M}_z$ on $H^2_\E(\D_2^{(\infty)})$ as its minimal isometric dilation, for some Hilbert space $\E.$ Equivalently, $\textbf{T}$ is unitarily equivalent to a sequence $(P_{\mathcal Q}M_{z_i}|_{\mathcal Q})_{i\geq 1}$ for some quotient module $\mathcal Q$ in $H^2_\E(\D_2^{(\infty)}).$

\end{enumerate}
\end{theorem}

The condition (1) in the above theorem helps us to find some easily verifiable sufficient condition on $\mathbf T$, which ensure that $\mathbf T$ has minimal isometric dilation on the Hardy space. 

\begin{proposition}\label{thm:summability}
Let $\mathbf T$ be a sequence of commuting pure contractions on $\mathcal H$ satisfying Brehmer's positivity. Suppose that there
is a dense subset $\mathcal F\subseteq\mathcal H$ such that
\begin{equation}\label{eq:vector-summability}
 \sum_{i=1}^\infty\|T_i^*h\|^2<\infty\qquad(h\in\mathcal F).
\end{equation}
Then $\textbf{T}$ has $\textbf{M}_z$ on $H^2_\E(\D_2^{(\infty)})$ as its minimal isometric dilation, for some Hilbert space $\E.$

In particular, the same conclusion holds if $\sum_i\|T_i\|^2<\infty$.
\end{proposition}
\begin{proof}
We show that if $\mathbf T$ satisfies the hypothesis, then it also satisfies condition (1) in Theorem~\ref{CharacterizationBeurlingtype}. To this end, first observe that, for any finite subset $F=\{j_1,j_2,\dots,j_r\}\subset\mathbb N$,
\[
I-\Delta_F(\mathbf T)
=\sum_{k=1}^r T_{j_k}\Delta_{F_{k-1}}(\mathbf T)T_{j_k}^*,
\]
where $F_k=\{j_1,j_2,\dots,j_k\}$ and $F_0=\emptyset$. Since $0\leq \Delta_G(\mathbf T)\leq I$ for every finite set $G$, it follows that
\[
I-\Delta_F(\mathbf T)
\leq \sum_{j\in F}T_jT_j^*.
\]
In particular,
\[
I-\Delta_{\{n+1,\dots,m\}}(\mathbf T)
\leq \sum_{n+1\leq j\leq m}T_jT_j^*.
\]
Consequently, for any $h\in \mathcal F,$
\[
\left\langle I-\T_{(n}(\mathbf T)h,h\right\rangle
\leq \sum_{j\geq n+1}\|T_j^*h\|^2.
\]
By the convergence hypothesis, it follows that
$$
\left\langle
\bigl(I-\T_{(n}(\mathbf T)\bigr)h,h
\right\rangle\longrightarrow 0
\qquad\text{as }n\to\infty,
$$
for every $h\in\mathcal F$. Since
$
0\leq I-\T_{(n}(\mathbf T)\leq I
$
for all $n$, the uniform boundedness of the operators, together with the density of $\mathcal F$ in $\mathcal H$, implies that
$$
I-\T_{(n}(\mathbf T)\longrightarrow 0
$$
in the weak operator topology. Finally, since
$
\bigl(I-\T_{(n}(\mathbf T)\bigr)^2
\leq I-\T_{(n}(\mathbf T),
$
we conclude that
$$
\sotlim_{n\to\infty} (I-\T_{(n}(\mathbf T))=0.
$$
This completes the proof.

\end{proof}
We work out an example to demonstrate the applicability of the results obtained in this section. 

\begin{example}\label{ex:nilpotent}
On $\mathcal H=\C^2$, with orthonormal basis $e_1,e_2$, consider
\[
 N=\begin{pmatrix}0&1\\0&0\end{pmatrix},\qquad
 a_i=\frac{\sqrt3}{2^{i}},\qquad T_i=a_iN\quad(i\ge1).
\]
Every $T_i$ is nonzero, pure, and strictly contractive. Since $T_iT_j=0$, $\mathbf T$ is commutative. The geometric series
gives $\sum_{i\ge1}a_i^2=1$, we have $\sum_i \|T_i\|^2<\infty$. 
Observe that, as $T_iT_j=0$, only subsets of size zero and one contribute to a finite Brehmer defect,
so
\[
 \Delta_F(\mathbf T)=I-\left(\sum_{i\in F}a_i^2\right)NN^*
 =\diag\left(1-\sum_{i\in F}a_i^2,1\right)>0.
\]
Thus $\mathbf T$ satisfies Brehmer's positivity. Therefore, by Proposition~\ref{thm:summability}, $\mathbf T$ has its minimal isometric dilation on the Hardy space. We compute the dilation map explicitly below.  Note that this sequence is not doubly commuting:
for $i\ne j$,
\[
 T_iT_j^*=a_i a_j\diag(1,0),\qquad
 T_j^*T_i=a_i a_j\diag(0,1).
\]
The defect operator and the tail defect operator are
\begin{equation}\label{eq:nilpotent-defects}
 D_{\mathbf T^*}=\diag(0,1),\qquad
 \T_{(n}(\mathbf T)=\diag(1-4^{-n},1)\quad(n\ge0),
\end{equation}
as $\sum_{i>n}a_i^2=4^{-n}$.
For $x=x_1e_1+x_2e_2\in \mathbb C^2$, we have
\[
 D_{\mathbf T^*} x=x_2e_2,\qquad
 D_{\mathbf T^*} T_i^*x=a_i x_1e_2,
 \qquad D_{\mathbf T^*} \mathbf T^{*\alpha}x=0\quad(|\alpha|\ge2).
\]
Consequently, one sees that 
\[
\sum_{\alpha \in \Z_+^{(\infty)}}\|D_{\textbf{T}^*}{\textbf{T}^*}^\alpha x\|^2= |x_2|^2+|x_1|^2(\sum_i a_i^2)=\|x\|^2.
\]
Identifying the defect space $\C e_2$ with $\C$, we define
\[
 g(z)=\sum_{i=1}^\infty a_i z_i.
\]
This function has Hardy norm one and is orthogonal to $1$. The map $\Pi_{\mathbf T}: \mathbb C^2\to H^2_{\mathbb C}(\D_2^{(\infty)})$ defined by 
\[
 \Pi_{\mathbf T}(x_1e_1+x_2e_2)=x_2+x_1g,\]
 is an isometry with $
 \operatorname{Ran}\Pi_{\mathbf T}=\mathcal Q=\operatorname{span}\{1,g\}\subset H^2_{\mathbb C}(\D_2^{(\infty)})$.
The identities $M_{z_i}^*g=a_i$ and $M_{z_i}^*1=0$ directly check the intertwining property $M_{z_i}^*\Pi_{\mathbf T}= \Pi_{\mathbf T} T_i^*$. Thus $\mathbf T$ has $\textbf{M}_z$ on $H^2_{\mathbb C}(\D_2^{(\infty)})$ as its minimal isometric dilation, that is, $\mathbf T \cong P_{\mathcal Q}\textbf{M}_z|_{\mathcal Q}$. The minimality follows from the fact that $1\in \mathcal Q$.
\end{example}

\section{Beurling quotient modules and model spaces }

In Theorem ~\ref{CharacterizationBeurlingtype}, we established a necessary and sufficient condition for a sequence $\textbf{T}$ of commuting contractions to be unitarily equivalent to $(P_{\mathcal Q}M_{z_i}|_{\mathcal Q})_{i\geq 1}$ for some quotient module $\mathcal Q$ in $H^2_\E(\D_2^{(\infty)}).$ In such a case, the quotient module $\mathcal Q$ is known as the model space for $\textbf T$. In this section, we characterize all such sequences $\textbf T$ for which the model spaces are Beurling quotient modules (Theorem \ref{CharacterizationdilationBeurlingquotientmodule}).  
Our characterization depends on a complete characterization of Beurling quotient modules of $H^2_\E(\D_2^{(\infty)})$, which is provided below. Let $\Omega$ denote either $\D_2^{(\infty)}$ or $\D^n$ for some $n\in \N.$ 

\begin{theorem}\label{compressionBeurlingquotientmodule}
Let $\mathcal Q$ be a quotient module of $H^2_\E(\Omega).$ Define $C_{j}:=P_\mathcal Q M_{z_j}|_{\mathcal Q}$ for all $j,$ where $P_\mathcal Q$ denotes the orthogonal projection of $H^2_\E(\Omega)$ onto $\mathcal Q.$  Then the following are equivalent.
\begin{enumerate}
\item[(1)] $\mathcal Q$ is a Beurling quotient module.

 \item[(2)] For all $i\neq j$, $$(I_\mathcal Q-C_{i}^*C_{i})(I_\mathcal Q-C_{j}^*C_{j})=0. $$ 

\item[(3)] For all $i\neq j$, $$(M_{z_j}|_{\mathcal Q^\perp})^*(M_{z_i}|_{\mathcal Q^\perp})=(M_{z_i}|_{\mathcal Q^\perp})(M_{z_j}|_{\mathcal Q^\perp})^*.$$
\end{enumerate}
\end{theorem}

\begin{proof}The equivalence between $(1)$ and $(3)$ is well-established in the literature; refer to \cite{Mandrekar}, \cite[Theorem 1.7]{Dan-Guo}, \cite[Theorem 1.3]{Sarkar-Wick}, and \cite[Theorem 5.1]{Olofsson}. For the case where $\Omega=\D^n$ for some $n\in \N$, the equivalence between $(2)$ and $(3)$ is proved in \cite{BKD}. We establish the equivalence between $(2)$ and $(3)$ for the case where $\Omega=\D_2^{(\infty)}.$ This is done using the result for $\D^n$ and identifying $H^2_\E(\D_2^{(\infty)})$ as a vector-valued Hardy space over the polydisc. 

Let $i\neq j.$ Let $N=max\{i,j\}.$ Let $$W_N:H^2_\E(\D_2^{(\infty)})\rightarrow H^2_{H^2_\E(\D_2^{(\infty)})}(\D^N)$$ be the unitary given by 
\[W_N(az_1^{\alpha_1}z_2^{\alpha_2}\cdots z_n^{\alpha_n})=\begin{cases}
az_1^{\alpha_1}z_2^{\alpha_2}\cdots z_n^{\alpha_n} &\text{if } n\leq N,\\
(az_1^{\alpha_{N+1}}z_2^{\alpha_{N+2}}\cdots z_{n-N}^{\alpha_n})z_1^{\alpha_1}z_2^{\alpha_2}\cdots z_N^{\alpha_N} & \text{if }  n>N
\end{cases}\]
for $a\in \E$ and $\alpha_m \in \Z_+.$
It is easy to check that $$W_NM_{z_k}W_N^*=M_{z_k}$$ for $k\in\{1,2,...,N\}.$ Here,
$M_{z_k}$ on the left side of the identity denotes multiplication by ${z_k}$ on $H^2_\E(\D_2^{(\infty)}),$ while $M_{z_k}$ on the right side denotes multiplication by ${z_k}$ on $H^2_{H^2_\E(\D_2^{(\infty)})}(\D^N).$ Note that $W_N(\mathcal Q)$ is a quotient module in $H^2_{H^2_\E(\D_2^{(\infty)})}(\D^N).$ Define 
$$D_{k}:=P_{W_N(\mathcal Q)}M_{z_k}|_{W_N(\mathcal Q)}$$ for  $k\in\{1,2,...,N\}.$ Let $W_N^{'}:\mathcal Q\to W_N(\mathcal Q)$ be the unitary given by $W_N^{'}=W_N|_\mathcal Q.$ Then for all $k=1,2,...,N$,
\begin{equation}\label{cdequality}
W_N^{'}C_{k}{W_N^{'}}^*=D_{k}, \text{ and }
(I_{W_N(\mathcal Q)}-D_{k}^*D_{k})=W_N^{'}(I_\mathcal Q-C_{k}^*C_{k}){W_N^{'}}^*.
\end{equation}
Also, if $W_N^{''}:\mathcal Q^\perp\to {W_N(\mathcal Q)}^\perp$ is the unitary given by $W_N^{''}=W_N|_{{\mathcal Q}^\perp}$, then 
\begin{equation}\label{uc}
W_N^{''}M_{z_k}|_{{\mathcal Q}^\perp}{W_N^{''}}^*=M_{z_k}|_{{W_N(\mathcal Q)}^\perp}
\end{equation} for all $k=1,2,...,N.$ 

Now, suppose $(2)$ is true. Then, by \cref{cdequality}, for any $k,s \in \{1,2,...,N\}$ with $k\neq s,$ we have $$(I_{W_N(\mathcal Q)}-D_{k}^*D_{k})(I_{W_N(\mathcal Q)}-D_{s}^*D_{s})=0.$$ Hence by \cite[Theorem 1.1]{BKD} and \cite[Theorem 1.3]{Sarkar-Wick}, \begin{equation*}
(M_{z_j}|_{{W_N(\mathcal Q)}^\perp})^*(M_{z_i}|_{{W_N(\mathcal Q)}^\perp})=(M_{z_i}|_{{W_N(\mathcal Q)}^\perp})(M_{z_j}|_{{W_N(\mathcal Q)}^\perp})^*.    
\end{equation*} Therefore, by \cref{uc}, we have $$(M_{z_j}|_{\mathcal Q^\perp})^*(M_{z_i}|_{\mathcal Q^\perp})=(M_{z_i}|_{\mathcal Q^\perp})(M_{z_j}|_{\mathcal Q^\perp})^*.$$ This proves $(3).$

Conversely, suppose $(3)$ is true. Then, by \cref{uc}, $$(M_{z_s}|_{{W_N(\mathcal Q)}^\perp})^*(M_{z_k}|_{{W_N(\mathcal Q)}^\perp})=(M_{z_k}|_{{W_N(\mathcal Q)}^\perp})(M_{z_s}|_{{W_N(\mathcal Q)}^\perp})^*$$ for all $k,s \in \{1,2,...,N\}$ with $k\neq s.$ Therefore, by \cite[Theorem 1.3]{Sarkar-Wick} and \cite[Theorem 1.1] {BKD}, $$(I_{W_N(\mathcal Q)}-D_{i}^*D_{i})(I_{W_N(\mathcal Q)}-D_{j}^*D_{j})=0.$$ Hence, by \cref{cdequality}, $$(I_\mathcal Q-C_{i}^*C_{i})(I_\mathcal Q-C_{j}^*C_{j})=0.$$ This completes the proof.
\end{proof}

As a particular case of the characterization of Beurling quotient modules above, let us consider a quotient module $\mathcal Q$ of $H^2_{\mathcal E}(\mathbb D_2^{(\infty)})$ which is of Beurling type and satisfies

$$
\operatorname{Ker}(I_{\mathcal Q}-C_1^*C_1)=\{0\}.
$$

It then follows from item $(2)$ in Theorem~\ref{compressionBeurlingquotientmodule} that
$$
I_{\mathcal Q}=C_j^*C_j
$$
for all $j\neq 1$. Thus each $C_j$ is an isometry for $j\geq 2$. Since $\mathcal Q$ is a quotient module, the norm-preserving property of $C_j$ is equivalent to $\mathcal Q$ being invariant under $M_{z_j}$ for every $j\geq 2$. Consequently, $\mathcal Q$ is a joint-reducing subspace for $(M_{z_2},M_{z_3},\dots)$ on $H^2_{H^2(z_1)\otimes\mathcal E }(z_2,z_3,\dots)$. Therefore, by \cite[Lemmas 2.11 and 2.10]{Dan-Guo},
\begin{equation}\label{onevariable factor}
\mathcal Q=\mathcal Q_0\otimes H^2_{\mathbb C}(z_2,z_3,\dots),
\end{equation}
where $\mathcal Q_0$ is a co-invariant subspace of $H^2_{\mathcal E}(\mathbb D)$, and $H^2_{\mathbb C}(z_2,z_3,\dots)$ denotes the scalar-valued Hardy space consisting of functions that are independent of the $z_1$ variable.

On the other hand, by the classical Beurling--Lax--Halmos theorem, a quotient module of the form~\eqref{onevariable factor} is a Beurling quotient module. More precisely, there exist a Hilbert space $\mathcal E'$ and an inner multiplier
$$
\Theta:\mathbb D\to\mathcal B(\mathcal E';\mathcal E),
$$
depending only on the variable $z_1$, such that
$$
\mathcal Q^\perp
=\Theta H^2_{\mathcal E'}(\mathbb D_2^{(\infty)}),
$$
where $\Theta$ is determined by the relation
$\mathcal Q_0^\perp=\Theta H^2_{\mathcal E'}(\mathbb D).
$
Thus we obtain the following one-variable rigidity result.

\begin{proposition}
Let $\mathcal Q$ be a quotient module of $H^2_{\mathcal E}(\mathbb D_2^{(\infty)})$ such that
$$
\operatorname{Ker}(I_{\mathcal Q}-C_1^*C_1)=\{0\},
$$
where, for each $i\geq 1$,
$$
C_i:=P_{\mathcal Q}M_{z_i}|_{\mathcal Q},
$$
and $P_{\mathcal Q}$ denotes the orthogonal projection of $H^2_{\mathcal E}(\mathbb D_2^{(\infty)})$ onto $\mathcal Q$. Then the following are equivalent:
\begin{enumerate}
\item[(1)] $\mathcal Q$ is a Beurling quotient module.

\item[(2)] $C_j$ is an isometry for all $j\geq 2$.

\item[(3)] There exists a quotient module $\mathcal Q_0$ of $H^2_{\mathcal E}(\mathbb D)$ such that

$$
\mathcal Q
=
\mathcal Q_0\otimes
H^2_{\mathbb C}(z_2,z_3,\dots).
$$

\item[(4)] There exist a Hilbert space $\mathcal E'$ and a one-variable inner multiplier
$$
\Theta:\mathbb D\to\mathcal B(\mathcal E';\mathcal E)
$$
such that
$$
\mathcal Q^\perp
=
\Theta H^2_{\mathcal E'}(\mathbb D_2^{(\infty)}).
$$
\end{enumerate}
\end{proposition}

We remark that the equivalence of $(2)$ and $(3)$ in the above proposition holds even without the hypothesis
$
\operatorname{Ker}(I_{\mathcal Q}-C_1^*C_1)=\{0\}.
$
We now present the main result of this section, which follows from Theorems \ref{CharacterizationBeurlingtype} and \ref{compressionBeurlingquotientmodule}.

\begin{theorem}\label{CharacterizationdilationBeurlingquotientmodule}
Let $\textbf{T}=(T_n)_{n\geq 1}$ be a sequence of commuting contractions on $\H.$ Then the following statements are equivalent.
\begin{enumerate}
\item[(1)] The sequence $\textbf{T}$ is unitarily equivalent to a sequence $(P_{\mathcal Q}M_{z_i}|_{\mathcal Q})_{i\geq 1}$ for some Beurling quotient module $\mathcal Q$ of $H^2_\E(\D_2^{(\infty)})$ with coefficient Hilbert space $\E$.

\item[(2)] $\textbf{T}$ satisfies Brehmer's positivity, the identity $$\|x\|^2=\sum_{\alpha \in \Z_+^{(\infty)}}\|D_{\textbf{T}^*}{\textbf{T}^*}^\alpha x\|^2$$ holds for all $x\in \H,$ and for all $i\neq j$, $$(I_\H-T_i^*T_i)(I_\H-T_j^*T_j)=0.$$ 
\end{enumerate}
\end{theorem}

In the remainder of this section, we prove that the class of doubly commuting submodules and the class of Brehmer submodules of $H^2_\E(\D_2^{(\infty)})$ are the same using an elementary fact observed in \cite{Yang}. This extends the result of \cite[Theorem 7.6]{Rani}, which states that a submodule of $H^2_\E(\D^n)$  is a Brehmer submodule if and only if it is a doubly commuting submodule. Recall that a submodule $\mathcal S$ of  $H^2_\E(\D_2^{(\infty)})$ is a Brehmer submodule if the sequence of isometries $\textbf{M}_z|_{\mathcal S}=(M_{z_1}|_{\mathcal S},M_{z_2}|_{\mathcal S},...)$ satisfies Brehmer's positivity.

\begin{theorem}\label{Brehmer-DC}
Let $\S$ be a submodule of $H^2_\E(\D_2^{(\infty)}).$ Then $\S$ is a Brehmer submodule if and only if it is a doubly commuting submodule.
\end{theorem}
\begin{proof}
It is obvious that if $\mathcal S$ is doubly commuting, that is, if the sequence of isometries $\textbf{M}_z|_{\mathcal S}=(M_{z_1}|_{\mathcal S},M_{z_2}|_{\mathcal S},...)$ is doubly commuting, then $\textbf{M}_z|_{\mathcal S}$ satisfies Brehmer's positivity.  Conversely, let $i,j\in \mathbb{N}$ and $i\neq j$ and set $R_k:=M_{z_k}|_{\mathcal S}$ for all $k\in\mathbb N$. By the hypothesis, the operator $I-R_iR_i^*-R_jR_j^*+R_iR_jR_j^*R_i^*$ is positive. Therefore, by \cite[Theorem 3.4]{Yang}, $R_i$ commutes with $R_j^*$ and hence $\textbf{M}_{z}|_{\mathcal S}$ is a doubly commuting sequence of isometries. This completes the proof. \end{proof}




\vspace{0.1in} \noindent\textbf{Conflict of interest:}
The authors state that there is no conflict of interest. No data sets were generated or analyzed during the current study.

\vspace{0.1in} \noindent\textbf{Acknowledgement:}
The first author is supported by the Mathematical Research Impact Centric Support (MATRICS) grant, File No: MTR/2021/000560, by the Science and Engineering Research Board (SERB), Department of Science \& Technology (DST), Government of India. The second named author is grateful to Dr. Chaman Kumar Sahu for some useful discussions and acknowledges the support of an Institute postdoctoral fellowship from the Indian Institute of Technology Bombay and a National Postdoctoral
Fellowship (N-PDF) from the ANRF, India
(File No. PDF/2025/006823).

\end{document}